\documentclass[10pt,a4paper]{amsart}
\usepackage{amscd,amssymb,amsopn,amsmath,amsthm,graphics,amsfonts,enumerate,verbatim,calc}
\usepackage[dvips]{graphicx}
\input xy
\xyoption{all}

\usepackage{centernot}
\usepackage{amssymb,amsmath}
\usepackage{mathpazo}

\headsep=1cm \numberwithin{equation}{section}
\newtheorem{theorem}{Theorem}[section]

\newtheorem{lemma}[theorem]{Lemma}
\newtheorem{proposition}[theorem]{Proposition}
\newtheorem{corollary}[theorem]{Corollary}
\newtheorem{claim}[theorem]{Claim}

\theoremstyle{definition}

\theoremstyle{remark}
\newtheorem{remark}[theorem]{Remark}

\newtheorem{example}[theorem]{Example}
\newtheorem{problem}[theorem]{Problem}
\newtheorem{fact}[theorem]{Fact}
\newtheorem{notation}[theorem]{Notation}

\newtheorem{acknowledgement}{Acknowledgement}

\newcommand{\Spec}{\operatorname{Spec}}

\newcommand{\nil}{\operatorname{nil}}

\newcommand{\rad}{\operatorname{rad}}

\newcommand{\kk}{\operatorname{k}}

\newcommand{\fm}{\frak{m}}

\newcommand{\q}{\frak{q}}

\newcommand{\Frac}{\operatorname{Frac}}
\newcommand{\Tr}{\operatorname{Tr}}
\newcommand{\length}{\operatorname{length}}
\newcommand{\Q}{\mathbb Q}
\newcommand{\F}{\mathbb F}
\newcommand{\m}{\mathfrak m}
\newcommand{\n}{\mathfrak n}

\begin{document}
	
	\author[]{Mohsen Asgharzadeh}
	
	\title[Flat radical pullback]{Descent along flat composed  with radicalization}
	
	\address{M. Asgharzadeh }
	\email{mohsenasgharzadeh@gmail.com}

	
\begin{abstract}
We study the following general situation. Let $R\to S$ be a finite flat morphism of regular rings of zero (resp. prime)  mixed or (equi)-characteristic. For  $P\in\Spec R$, put $A=R/P,  C=S/PS,$ and $B=S/\sqrt{PS}=C_{\mathrm{red}}.$
The basic question is whether a property of $B$ forces the same property
of $A$. The main point is that ordinary finite-flat descent applies naturally to $A\to C$, whereas the passage $C\to C_{\mathrm{red}}$ may (not) destroy nilpotent information. 
\end{abstract}
\maketitle

\section{Introduction}

A recurring problem in commutative algebra is to understand which
properties descend along finite flat homomorphisms.  Let
\(
R\to S
\)
be a finite flat homomorphism and let \(P\in \operatorname{Spec}(R)\).
A natural question is whether properties of the ring
\(
S/\sqrt{PS}
\)
force the corresponding properties of
\(
R/P.
\)
Although descent along finite flat maps is well understood for many
algebraic properties, the passage from \(S/PS\) to its reduction
\(S/\sqrt{PS}\) introduces a new difficulty: it involves a nilpotent
thickening.
To make this distinction clear, throughout the paper we consider the
diagram
\[
R/P \longrightarrow S/PS \longrightarrow S/\sqrt{PS}.
\]
The first map is finite flat, while the second map is a quotient by the
nilradical.  Thus, questions involving \(S/\sqrt{PS}\) naturally
separate into two different problems:
descent along finite flat morphisms and stability under nilpotent
extensions.
A motivating example is the following problem of Lipman.
\begin{problem}(J. Lipman \cite{j}).\label{1.1}
	Let
	\(
	R\to S
	\)
	be an injective map of regular rings of  characteristic zero such that $S$ is finite and (flat) over $R$, and let $P\subset R$ be prime. If $S/\sqrt{PS}$ is regular, then is $R/P$ regular?
\end{problem} We present the following:
\begin{theorem}\label{142}
	Let
	\(
	R=\boldsymbol{W}(\mathbb F_2)[[x,y]]\),
	\(
	S=\boldsymbol{W}(\mathbb F_2)[[t,y]],
	\)
	and define
	\(\varphi:R\to S\)
	by
	\(
	\varphi(x)=t^2, 
	\varphi(y)=y.
	\)
	Let
	\(
	P=(2,\ y^2-x^3)\subset R.
	\)
	Then:
	\begin{enumerate}
		\item \(R\) and \(S\) are complete regular local rings of characteristic zero;
		\item \(S\) is finite free of rank \(2\) over \(R\) and \(P\) is prime;

		\item
		\(
		S/\sqrt{PS}
		\)
		is regular and
		\(
		R/P
		\)
		is not regular.
	\end{enumerate}

\end{theorem}

The corresponding question in equal characteristic seems to be considerably more subtle than Problem \ref{1.1}, at least for current author.
\begin{problem}(Huneke \cite[Question 2.3]{hfind}).\label{1.1.1}
	Let
	\(
	R\to S
	\)
	be an injective map of regular rings of equicharacteristic characteristic zero such that $S$ is finite and (flat) over $R$, and let $P\subset R$ be prime. If $S/\sqrt{PS}$ is regular, then is $R/P$ regular?
\end{problem}

\begin{remark}
Since \(S\) is module-finite over \(R\), the extension \(R\subseteq S\) is integral. By the going-up theorem, or equivalently by the dimension theorem for integral extensions,
\(
\dim(S/\mathfrak mS)=\dim S/\mathfrak mS=0,
\)
because every prime ideal of \(S\) containing \(\mathfrak mS\) contracts to a prime ideal of \(R\) containing \(\mathfrak m\). Since \(R\) is local, that contraction must be \(\mathfrak m\). By in comparability, two comparable prime ideals of \(S\) cannot have the same contraction. According to \cite[23.1]{mat} $S$ is $R$-flat.\end{remark}
 We present the following:
\begin{theorem}\label{152}
 	Let $(R,\m)\to(S,\n)$ be a finite flat local homomorphism of
 	Noetherian local rings. Suppose that $S$ is free of prime rank
 	$\ell$ over $R$ and that $\ell\in R^\times$. For every
 	$P\in\Spec R$, set
 	\[
 	A=R/P,\qquad C=S/PS,\qquad
 	B=S/\sqrt{PS}=C_{\mathrm{red}}.
 	\]
 	Then
 	\[
 	B\text{ is  regular}\quad\Longrightarrow\quad A\text{ is  regular}.
 	\]
 	In particular, this holds for every prime-rank finite flat local
 	homomorphism of Noetherian local $\Q$-algebras. Regularity of the
 	ambient rings $R$ and $S$ is not required.
\end{theorem}

 Let $\mathcal P$ be a property of rings, say for example the regularity.
An slightly more general question appears, namely suppose $B$   has $\mathcal P$. When is
\( A \) satisfy in
\(\mathcal P?
\)

The positive results of the paper explain when this obstruction does
not occur. We show that several finiteness and Frobenius-theoretic
properties descend through the above construction. In particular, if
\(
S/\sqrt{PS}
\)
is coherent (respectively \(F\)-coherent), then under suitable
hypotheses the same holds for \(R/P\). According to \cite{A}, and despite their differences,
\(F\)-coherent rings are very close to weakly \(F\)-nilpotent rings.
We therefore continue this philosophy and prove the analogous results
for weakly \(F\)-nilpotent rings. For rings of prime characteristic, we further study the stronger
property of \(F\)-nilpotence. Using the characterization of
\(F\)-nilpotent rings in terms of perfect closures and balanced big
Cohen--Macaulay algebras, we show that \(F\)-nilpotence descends under
finite weakly étale local homomorphisms. The main technical point is
the compatibility of perfect closures and the associated balanced big
Cohen--Macaulay algebras with weakly étale base change.

The results indicate that radicalization is a fundamental obstruction
to descent. For properties stable under nilpotent thickenings, the
problem reduces to ordinary finite flat descent, while regularity shows
that such stability cannot be expected in general.

The paper is organized as follows. In Section~2 we study properties
preserved under radicalization and finite flat descent, including
coherence and Frobenius-theoretic properties. 
Over normal homomorphisms  things are so beautiful, see
Proposition \ref{n}.
In Section~3 we discuss
further consequences concerning dimension-theoretic properties, and we give negative answers for several properties, including generalized Cohen--Macaulayness, the Buchsbaum property, complete intersection, and
\(F\)-purity (\(F\)-rationality, and related topics).
In
Section~4 we prove {Theorem} \ref{142}.  We combine the method of {Theorem} \ref{142} along with $\mathcal P:=\{complete-intersection\}$ from Section 3, and  study the corresponding descent approach to $\mathcal P:=\{Gorensteinness\}$.
In the final section  we prove {Theorem} \ref{152}, and present few additional related results.

\section{Some well behaved properties}
 We cite the books \cite{BH} and \cite{MP} for basic definitions that we need.
 \begin{notation}Let \(R\to S\) be flat. Set
 	$
 	A=R/P,  C=S/PS$ and
 	$B=S/\sqrt{PS}.
 	$ \end{notation}
 We first record several finiteness properties which behave well with respect to
 nilpotent thickenings.
 
 \begin{proposition}
 	Let \(A\) be a  ring, \(N\subseteq A\) be a nilpotent
 	ideal, and assume it is finitely generated. Then
 $
 	A\text{ is noetherian}$ iff
 	$A/N$ { is noetherian}.
 Moreover, if \(A\) is artinian, then \(A/N\) is artinian, and conversely
 	if \(A/N\) is artinian, then \(A\) is artinian.
 \end{proposition}
 
 \begin{proof}
 	The first assertion is standard. If \(A\) is noetherian, then every
 	quotient of \(A\) is noetherian. Conversely, suppose that \(A/N\) is
noetherian. Since \(N\) is nilpotent, there is an integer \(t\) such
 	that
  $
 	N^t=0.
 	$ Consider the finite filtration
 $
 	A\supseteq N\supseteq N^2\supseteq\cdots\supseteq N^t=0.
 	$  Each quotient \(N^i/N^{i+1}\) is naturally an \(A/N\)-module. Since
 	\(N^i/N^{i+1}\) is finitely generated as an \(A/N\)-module, it is
 	noetherian. It follows inductively from the filtration that \(A\) is
 	noetherian.
 	
 	For the artinian assertion, the same filtration shows that if
 	\(A/N\) is artinian, then each
  $
 	N^i/N^{i+1}
 	$ is an artinian \(A/N\)-module. Hence \(A\) is artinian. The converse
 	follows since quotients of artinian rings are artinian.
 \end{proof}
The finiteness assumption on $N$ is really needed.
\begin{example}
	Let $A:=\kk[X_n:n\in \mathbb{N}]/(X_n^2:n\in \mathbb{N})$, where $k$ is a field. Then $A$ is not noetherian while 
	$A/ \nil(A)=\frac{\kk[X_n:n\in \mathbb{N}]/(X_n^2:n\in \mathbb{N})}{(X_n:n\in \mathbb{N})/(X_n^2:n\in \mathbb{N})}=\kk,$ which is noetherian.
\end{example}

\begin{corollary}
Let $R\to S$ be flat morphism of quasilocal rings, and assume \(\sqrt{PS}/PS\) is finitely generated.
Then  
\[
S/\sqrt{PS}\text{ is noetherian }\Longrightarrow\quad R/P\text{ is noetherian.} 
\]	
\end{corollary}

\begin{proof}
By previous proposition $S/ {PS}$ is noetherian. Recall that $R/P\to S/ {PS}$ is faithfully flat. Now, we apply \cite[Ex. 7.9]{mat} to see $R/P$  is noetherian. 
\end{proof}

\begin{proposition}
	Let \(A\) be a ring and let \(N\subseteq A\) be a nilpotent ideal. Assume
	that \(A/N\) is coherent and that \(N\) is a coherent \(A\)-module. Then
	\(A\) is coherent.
\end{proposition}

\begin{proof}
	Consider the exact sequence
$
	0\to N\to A\to A/N
	\to 0.
	$ 	Since \(N\) is a coherent \(A\)-module and \(A/N\) is a coherent ring,
	the middle term \(A\) is a coherent \(A\)-module. Hence \(A\) is a
	coherent ring.
\end{proof}

\begin{proposition}
	Let \(R\to S\) be a finite faithfully flat homomorphism of
	rings, let \(P\in\operatorname{Spec}(R)\), and set
 $
	N=\frac{\sqrt{PS}}{PS}.
	$ 	Assume that \(S/\sqrt{PS}\) is coherent and that \(N\) is a coherent
	\(S/PS\)-module. Then \(R/P\) is coherent.
\end{proposition}

\begin{proof}
	There is an exact sequence	
 $
	0\to
	\frac{\sqrt{PS}}{PS}
	\to
	\frac{S}{PS}
	\to
	\frac{S}{\sqrt{PS}}
	\to 0.
	$ By assumption,
$
	\frac{\sqrt{PS}}{PS}
	$ is a coherent \(S/PS\)-module, while
$
	S/\sqrt{PS}
	$ is a coherent ring. Hence, by the preceding proposition,
$
	S/PS
	$ is coherent.
	Since \(R\to S\) is faithfully flat, the induced map
$
	R/P\to S/PS
	$ 	is faithfully flat. Coherence descends along faithfully flat
	homomorphisms. Therefore \(R/P\) is coherent.
\end{proof}
The coherence is not invariant under nilpotent
thickenings $
A/N\text{ coherent}
\centernot\Longrightarrow
A\text{ coherent}:
$
\begin{example}
	Let \(\kk\) be a field and let \(V\) be an infinite-dimensional
	\(k\)-vector space. Consider the trivial extension
$
	A=\kk\ltimes V,
	$ with multiplication
$
	(a,v)(b,w)=(ab,aw+bv).
$
Set
$
	N=0\ltimes V.
	$ Then
 $
	N^2=0, A/N\cong k,
	$ so \(A/N\) is coherent, but \(A\) is not coherent.
	Moreover, \(N\) is not a coherent \(A\)-module.
\end{example}

\begin{proof}
	Clearly
$
	N^2=0
	$ and
$
	A/N\cong k.
	$ Thus \(A/N\) is a field, and in particular it is coherent.
 We claim that \(A\) is not coherent. Choose a nonzero element
	\(v\in V\), and let
 $
	I=Av.
	$
Since \(N^2=0\), multiplication by an element
	\((a,w)\in A\) on \(v=(0,v)\) is given by
$
	(a,w)(0,v)=(0,av).
$
Thus
$
	I\cong kv
$
as an \(A\)-module, and \(I\) is finitely generated.
Now consider the annihilator
$
	\operatorname{Ann}_A(v)
	=
	\{(a,w)\in A:av=0\}.
	$	Since \(V\) is a \(k\)-vector space and \(v\neq0\), this gives
$
	\operatorname{Ann}_A(v)=0\ltimes V=N.
	$	If \(A\) were coherent, the annihilator of every element would be a
	finitely generated ideal. Hence \(N\) would be finitely generated.
But \(N=0\ltimes V\) is finitely generated as an \(A\)-module if and
	only if \(V\) is finite-dimensional over \(k\), contrary to our
	assumption. Therefore \(A\) is not coherent.
	In particular, the nilpotent ideal \(N\) is not a coherent
	\(A\)-module.
\end{proof}

From now on all rings are assumed  to be noetherian and local, otherwise we specialized.
Over normal homomorphisms (see \cite[Definition 19.0.1]{HS} for a definition) things are so beautiful.
\begin{proposition}\label{n}
	Let \(R\to S\) be a normal homomorphism of rings. Then Problem \ref{1.1} holds for the following properties$$\mathcal P
	\in\{regularity; Cohen-Macaulay; complete-intersection; Gorenstein; normality,...\}.$$
\end{proposition}

\begin{proof} Let
	\(P\in\operatorname{Spec}(R)\). Set
	$A=R/P$ and $C=S/PS.$ By \cite[Proposition 19.1.2]{HS} $A\to C$ is normal. Since $A$ is an integral domain ($P$ is prime) and
	$A\to C$ is normal, by \cite[Proposition 19.1.3~(1)]{HS} we deduce that $C$ is reduced. Then it follows that
$
	S/PS=S/\sqrt{PS}=:B.
 $
	Thus \(A\to B\) is  faithfully flat.
	This gives all desired properties. For instance, \cite[Theorem 2.2.12]{BH} (resp. \cite[Theorem  2.1.7]{BH}; and  \cite[Corollary 3.3.15]{BH}) shows regularity (resp. Cohen-Macaulay; and Gorenstein) of $A$ descent from the regularity (resp. Cohen-Macaulay; and Gorenstein) of $B$.
\end{proof}

In the previous item, primeness of $P$ was so essential.
In the next item,  we deal with ideals that  are not assumed to be prime.
Clearly the unmixed property is not the case, as there are several well-known example of ideals as $I\neq I^{unmixed}$ but $\rad(I)$ is equal to its unmixed part. 
\begin{proposition}
	The assumption  $S/\sqrt{IS}\text{ is equidimensional overits  minimal primes}
	$ implies that
	$R/I$ is equidimensional over $\min(I)$.
\end{proposition}

\begin{proof}
	Since
	$
	\operatorname{Spec}(S/IS)
	\cong
	\operatorname{Spec}(S/\sqrt{IS}),
	$ the rings \(S/ IS\) and \(S/\sqrt{IS}\) have the same prime spectrum.
	In particular, they have the same irreducible components and the same
	Krull dimension. Hence
	
	$$(\ast):\quad
	S/\sqrt{IS}\text{ is equidimensional  over its min}
	\quad\Longrightarrow\quad
	S/IS\text{ is equidimensional over its min}.
	$$
Let $\frak q\in V(I)$ be minimal. By lying over for integral extensions (see \cite[Theorem 9.3]{mat}), there is $Q\in V(IS)$ so that $Q\cap R=\frak q$. Let $Q'$ be a prime so that $IS\subseteq Q' \subseteq Q$. By faithfully flatness $I=IS\cap R= Q'	\cap R\subseteq Q	\cap R=\frak q$. By minimality of $\frak q$, $Q'	\cap R= Q	\cap R=\frak q$. By incomparability for lying over (see \cite[theorem 9.3.(ii)]{mat}) we are able to deduce that $Q=Q'$. So, $Q\in \min(IS)$. Thus, $\dim(S/Q)=\dim(S/IS)$ by  $(\ast)$. But
  \( R/I\to  S/IS\) is finite faithfully flat. Since the fibers of a
	finite morphism have dimension zero,  we have $$\dim(S/Q)=\dim(S/IS)=\dim(R/I)\geq \dim(R/\frak q)=\dim(S/\frak q S)= \dim(S/Q),$$ where the last two equalities follow because
  \(R/\frak q\to  S/\frak q S\) is also finite faithfully flat, $Q\in \min(\frak q S)$ and by $(\ast)$.	
	Therefore $\dim(R/I)= \dim(R/\frak q)$
	and so \( R/I\) is
	equidimensional over its minimal prime ideals.
\end{proof}

 	Let \(R\) be a ring of characteristic \(p>0\) equipped with Frobenius map $F$. The \emph{perfect closure} of \(R\)
is the colimit
\(
R^{\mathrm{perf}}
:=R^\infty:=
\varinjlim_{e\in\mathbb N} F_*^e R,
\)
where the transition maps \(F_*^e R\to F_*^{e+1}R\) are induced by the Frobenius
\(F\colon R\to R\).
Recall from \cite{Shi10} that $A$ is called \emph{$F$-coherent} if \(A^{\mathrm{perf}}\) is coherent.
The following observation is useful in the presence of nilpotents.

\begin{lemma}
	Let \(A\) be a ring of characteristic \(p>0\). Then
	$A^{\mathrm{perf}} \cong (A_{\mathrm{red}})^{\mathrm{perf}}.
	$
	Consequently,
	\[
	A\text{ is $F$-coherent}
	\quad\Longleftrightarrow\quad
	A_{\mathrm{red}}\text{ is $F$-coherent}.
	\]
\end{lemma}

\begin{proof}
	Let \(N=\sqrt{(0)}\) be the nilradical of \(A\). Since every element of
	\(N\) is nilpotent, for each \(x\in N\) there exists \(e\geq 0\) such that
	\(x^{p^e}=0\). Thus \(N\) is killed in the direct limit defining the
	perfect closure. Hence the canonical map
	\(A^\infty\to (A/N)^{\mathrm{perf}}\)
	is an isomorphism. Since \(A/N=A_{\mathrm{red}}\), the assertion follows.
\end{proof}

\begin{corollary}
	Let \(R\to S\) be a ring homomorphism of characteristic $(p>0)$, and let
	\(P\in\operatorname{Spec}(R)\). Then
	$
	(S/PS)^{\mathrm{perf}}
	\cong
	\bigl(S/\sqrt{PS}\bigr)^{\mathrm{perf}}.
	$
	In particular,
	\[
	S/PS\text{ is $F$-coherent}
	\quad\Longleftrightarrow\quad
	S/\sqrt{PS}\text{ is $F$-coherent}.
	\]
\end{corollary}

\begin{proof}
	The nilradical of \(S/PS\) is
	$
	\sqrt{PS}/PS.
	$
	Hence
	$
	(S/PS)_{\mathrm{red}}
	\cong S/\sqrt{PS}.
	$
	The assertion now follows from the previous lemma.
\end{proof}

Let us present a property weaker than F-coherent rings. Recall  that $R$ is called \emph{weakly F-nilpotent}  if  every element of $H^i_\fm(R)$ is annihilated by some   $F^n$ for
$i < \dim(R)$  (see \cite{A} for their differences and some comparisons). 
\begin{proposition}\label{212}
	Let \(R\to S\) be a finite  flat homomorphism of
	local rings of characteristic \(p>0\), let
	\(P\in\operatorname{Spec}(R)\).
 If \(S/\sqrt{PS}\) is weakly
	\(F\)-nilpotent, then \(R/P\) is weakly \(F\)-nilpotent.
\end{proposition}
The proposition  follows easily 
theorem from \cite[Proposition 12.22]{MP}, 
 namely a characterization of weakly \(F\)-nilpotent in terms of big Cohen-Macaulayness of its perfect closure. 
 
\begin{proof}
 Put
$
	N=\sqrt{PS}/PS\subseteq C.
	$ Then \(N\) is the nilradical of \(C\), and
$
	C/N\cong B.
$
There exists \(e\geq 1\) such that
$
	N^{p^e}=0.
	$	In particular, the Frobenius action on \(N\) is nilpotent. Hence the
	Frobenius action on each local cohomology module
	\(H^i_{\mathfrak m_C}(N)\) is nilpotent.
	From the exact sequence
$
	0\to N\to C\to B
	\to0
	$ we obtain the Frobenius-compatible exact sequence
 $
	\cdots\to
	H^i_{\mathfrak m_C}(N)
\to
	H^i_{\mathfrak m_C}(C)
\to
	H^i_{\mathfrak m_B}(B)
\to
	H^{i+1}_{\mathfrak m_C}(N)
\to\cdots .
	$ 
	Since the Frobenius action on the two terms involving \(N\) is
	nilpotent, and \(B\) is weakly \(F\)-nilpotent, it follows that
	every element of
$
	H^i_{\mathfrak m_C}(C)$ with $ i<\dim C,
$
	is killed by some power of Frobenius. Thus \(C\) is weakly
	\(F\)-nilpotent.
Now \(A\to C\) is faithfully flat. Since \(A\to C\) is finite,
	we have
$
	\dim A=\dim C.
$	Moreover, local cohomology commutes with flat base change, so
$
	H^i_{\mathfrak m_A}(A)\otimes_A C
	\cong
	H^i_{\mathfrak m_C}(C)
	$
for every \(i\). Let
$
	\eta\in H^i_{\mathfrak m_A}(A)$, $i<\dim A.
	$ Since \(C\) is weakly \(F\)-nilpotent, the element
	\(\eta\otimes1\) is annihilated by some power \(F^n\). Hence
$
	F^n(\eta)\otimes1=0.
$
Because \(A\to C\) is faithfully flat, the map
$
	H^i_{\mathfrak m_A}(A)
	\to
	H^i_{\mathfrak m_A}(A)\otimes_A C
	$ is injective. Therefore
$
	F^n(\eta)=0.
	$ Thus every element of \(H^i_{\mathfrak m_A}(A)\), for
	\(i<\dim A\), is annihilated by some power of Frobenius. Hence
	\(A=R/P\) is weakly \(F\)-nilpotent.
\end{proof}

Let us present third proof of the above item: \begin{remark}By \cite[Corollary 5.9]{A} weakly \(F\)-nilpotent corresponds to cohomological complete-intersection. It remain to note that
we have flat base change for local cohomology modules and also $H^i_{I}(-)=H^i_{\sqrt{I}}(-)$. Both are basic and well-known to every one.\end{remark}
 
For definition of \(F\)-nilpotent, see \cite[\S 12]{MP}.
Here, we record some basic facts from \cite{MP}.

\begin{remark} Let \(R\) be a reduced excellent local ring of characteristic \(p>0\). Let
	\(
	\prod_{\mathbb N}R^{\mathrm{perf}}
	\)
	be the product of countably many copies of \(R^{\mathrm{perf}}\), indexed by
	\(\mathbb N\), equipped with the diagonal Frobenius action induced componentwise
	by \(F\) on \(R^{\mathrm{perf}}\). Let
	\(
	W\subseteq \prod_{\mathbb N}R^{\mathrm{perf}}
	\)
	be the multiplicative set generated by
	\(
	(c,F_*c,F_*^2c,\dots)
	\)
	for all \(c\in R\) not contained in any minimal prime of \(R\), embedded
	diagonally into \(\prod_{\mathbb N}R^{\mathrm{perf}}\). The
	\emph{ algebra associated to \(R\)} is
	\( 
	\mathcal B(R)
	:=
	W^{-1}\prod_{\mathbb N}R^{\mathrm{perf}}.
	\)
When the ring is a homomorphic image of a Cohen-Macaulay ring
and in view of \cite[Theorem 12.15]{MP} we have

	\[
R \text{ is equidimensional} \Longleftrightarrow \mathcal B(R)\text{  is balanced big Cohen--Macaulay.}
\]
\end{remark}

\begin{proposition} 
	\label{thm:Fnilpotent-descent}
	Let \((R,\mathfrak m)\to(S,\mathfrak n)\) be a finite weakly etale local homomorphism of
	excellent local rings of characteristic \(p>0\).
	Assume that \(F\)-nilpotence descends along weakly etale local homomorphisms,
	i.e.\ whenever \((A,\mathfrak a)\to(B,\mathfrak b)\) is a weakly etale local
	homomorphism of excellent local rings of characteristic \(p>0\) with
 and \(B\) is \(F\)-nilpotent, then \(A\) is
	\(F\)-nilpotent.
\end{proposition}

\begin{proof}
	Since \(R\to S\) is weakly etale and
	\(
	S\otimes_R R/P\cong S/PS,
	\)
	the base change
	\(
	R/P\to S/PS
	\)
	is weakly etale. Indeed, weakly etale morphisms are stable under arbitrary base
	change. Thus
	\(
	A=R/P\to C=S/PS
	\)
	is a weakly etale local homomorphism 
	so the local hypothesis is preserved.
	Let
	\(
	I:=\sqrt{PS}/PS\subseteq C.
	\)
	Then \(I\) is a nilpotent ideal of \(C\), and
	\(
	C/I\cong S/\sqrt{PS}=B.
	\)
	Since \(B\) is \(F\)-nilpotent by hypothesis, and \(F\)-nilpotence is insensitive
	to nilpotent thickenings (see the method of \ref{212}), it follows that
	\(
	C=S/PS
	\)
	is \(F\)-nilpotent.
	We now have a weakly etale local homomorphism
	\(
	A=R/P\to C=S/PS
	\)
and \(C\) is
	\(F\)-nilpotent. 
Recall that \(C\) commutes with products in
	the category of \(A\)-modules:
	\(
	C\otimes_A \prod_{\mathbb N}A^{\mathrm{perf}}
	\cong
	\prod_{\mathbb N}\bigl(C\otimes_A A^{\mathrm{perf}}\bigr).
	\) because it is finitely presented.
Moreover, as \(A\to C\) is weakly etale, then Frobenius is compatible with base
	change:
	\(
	C\otimes_A F_*^e A
	\cong
	F_*^e C,
	\)
	so that
	\(
	C^{\mathrm{perf}}
	\cong
	A^{\mathrm{perf}}\otimes_A C,
	\)
	and hence
	\(
	C\otimes_A \prod_{\mathbb N}A^{\mathrm{perf}}
	\cong
	\prod_{\mathbb N}C^{\mathrm{perf}}.
	\)
	Since \(A\) is a domain, the elements defining \(W_A\) are the Frobenius
	orbits of nonzero elements of \(A\). Let \(0\neq a\in A\). The exact sequence
	\(
	0\to A\xrightarrow{a}A
	\)
	remains exact after tensoring with the flat \(A\)-module \(C\), hence
	multiplication by \(a\) is injective on \(C\). Therefore \(a\) remains a
	non-zero-divisor in \(C\).
	By weak etaleness, Frobenius commutes with base change, and hence the image
	of
	\(
	(a,F_*a,F_*^2a,\ldots)
	\)
	is exactly the Frobenius orbit of the image of \(a\) in \(C\). Thus every
	generator of \(W_A\) maps into \(W_C\), giving 
	the image of \(W_A\) in
	\(\prod_{\mathbb N}C^{\mathrm{perf}}\)
	is contained in \(W_C\). This gives the canonical map
	\(
	W_A^{-1}\prod_{\mathbb N}C^{\mathrm{perf}}
	\to
	W_C^{-1}\prod_{\mathbb N}C^{\mathrm{perf}}
	\) fits in the following diagram
	$$\xymatrix{
		&&  W_A^{-1}\prod_{\mathbb N}C^{\mathrm{perf}} \ar[r]^{}\ar[d]_{=}&W_C^{-1}\prod_{\mathbb N}C^{\mathrm{perf}}
		\ar[d]^{{=}}\\
		&& B(A)\otimes_A C\ar[r]^{}& B(C) 
		&&&}$$
and so the following composition$$(\ast):\quad\xymatrix{
	&B(A)\ar[r]^{}\ar[d]_{\exists}&B(A)\otimes_A C\ar[dl]^{}\\
	&  B(C) 
	&&&}$$

Indeed, let \(d=\dim A=\dim C \). Since \(C\) is \(F\)-nilpotent,   \cite[Proposition 12.23]{MP} gives us:
	\begin{enumerate}
		\item \(C^{\mathrm{perf}}\) is a balanced big Cohen--Macaulay algebra;
		\item the map
		\(
		H^d_{\mathfrak c}(C^{\mathrm{perf}})\to
		H^d_{\mathfrak c}(\mathcal B(C))
		\)
		is injective.
	\end{enumerate}

	Using the isomorphism
	\(
	C^{\mathrm{perf}}\cong A^{\mathrm{perf}}\otimes_A C
	\)
	and the fact that \(A\to C\) is weakly etale, one checks that the balanced big
	Cohen--Macaulay property descends from \(C^{\mathrm{perf}}\) to
	\(A^{\mathrm{perf}}\). Indeed, a system of parameters of \(A\) maps to a system of
	parameters of \(C\), and regular sequences descend along faithfully flat maps.
	The balanced condition is preserved because the \(F\)-orbits are compatible with
	weakly etale base change by previous remark, and because the
	product \(\prod_{\mathbb N}(-)\) commutes with finitely presented base change by the same
remark.
Let
\(
x_1,\ldots,x_d
\)
be a system of parameters of \(A\). Its image is a system of parameters of
\(C\). Since \(C^{\mathrm{perf}}\) is balanced big Cohen--Macaulay, this
sequence is regular on \(C^{\mathrm{perf}}\). Faithful flat descent of
regular sequences implies that
\(
x_1,\ldots,x_d
\)
is regular on \(A^{\mathrm{perf}}\). Hence \(A^{\mathrm{perf}}\) is balanced
big Cohen--Macaulay.
Recall  that $M\subset M\otimes_AC$, and by flat base change for local cohomology 	$$H^d_{\mathfrak a}(A^{\mathrm{perf}})\subset H^d_{\mathfrak a}(A^{\mathrm{perf}})\otimes C\cong H^d_{\mathfrak a}(A^{\mathrm{perf}}\otimes C)=
H^d_{\mathfrak a}(C^{\mathrm{perf}}).$$	
By the functoriality of local cohomology and the independent theorem for local cohomology applied to $(\ast)$
we have 
the  
commutative diagram
	\[
	\begin{array}{ccc}	0 & 
  \\
	\downarrow\\
	H^d_{\mathfrak a}(A^{\mathrm{perf}}) & \stackrel{f }\longrightarrow &
	H^d_{\mathfrak a}(\mathcal B(A)) \\[1mm]
	\downarrow & & \downarrow \\[1mm]
0\longrightarrow	H^d_{\mathfrak c}(C^{\mathrm{perf}}) & \longrightarrow &
	H^d_{\mathfrak c}(\mathcal B(C)),
	\end{array}
	\]
  shows	that $f$
  	is injective. Since \(A^{\mathrm{perf}}\) is a balanced big Cohen--Macaulay algebra,
 applying  \cite[Proposition 12.23]{MP}   again, \(A\) is
	\(F\)-nilpotent.
\end{proof}

\section{Failure of descent}

This section is divided into two subsections. 
The first one is about descent of mild singularities like Cohen--Macaulay and complete-intersection. The second subsection
is about F-singularities such as F-purity, F-regularity and related topics.

\subsection{(Generalized) Cohen--Macaulay, Buchsbaum, and complete-intersection}
For generalized Cohen--Macaulayness and the Buchsbaum property, the
situation is different. These properties need not be preserved when
one passes from a reduced ring to a nilpotent thickening.

\begin{example}
	Let \(\kk\) be a field and put
 $
	A=\kk[[x,y,z]]
	$ 	and
 $
	N=A/(x,y).
	$ 	Consider the idealization
 $
	B=A\ltimes N.
	$ Thus, as an \(A\)-module,
 $
	B=A\oplus N,
 $
 with multiplication
 $
	(a,n)(a',n')=(aa',an'+a'n).
 $
 In particular,
 $
	(0\oplus N)^2=0.
 $
 	Consequently,
 $
	\sqrt{0_B}=0\oplus N
	$ 
and
 $
	B_{\mathrm{red}}
	\cong B/(0\oplus N)
	\cong A
	\cong \kk[[x,y,z]].
	$ Hence \(B_{\mathrm{red}}\) is regular, and therefore Cohen--Macaulay,
	generalized Cohen--Macaulay, and Buchsbaum.
 Nevertheless, \(B\) is neither generalized Cohen--Macaulay nor
	Buchsbaum.
\end{example}

\begin{proof}
	The ring \(B\) is local with maximal ideal
 \(\fm_B\) is written as \((x,y,z)A \oplus N\).
 	Moreover,
 $
	\dim B=\dim A=3.
	$ 	Let
 $
	\mathfrak p=(x,y)A \oplus N$, the prime of \(B\) lying over \((x,y)A\), and .	Since \(N=A/(x,y)\), we have
 $
	N_{\mathfrak p}\neq 0,
 $
 and \(x\) and \(y\) annihilate \(N_{\mathfrak p}\). Thus
 $
	\operatorname{depth}N_{\mathfrak p}=0.
	$ 	From the exact sequence
 $
	0\to N\to B\to A\to0
	$ we obtain, after localizing at \(\mathfrak p\),
$
	0\to N_{\mathfrak p}
	\to B_{\mathfrak p}
	\to A_{(x,y)}
	\to0.
	$ Since
 $
	\operatorname{depth}N_{\mathfrak p}=0
 $
 and \(A_{(x,y)}\) has positive depth, it follows that
 $
	\operatorname{depth}B_{\mathfrak p}=0.
	$ On the other hand,
 $
	\dim B_{\mathfrak p}
	=
	\dim A_{(x,y)}
	=
	2.
	$ 
 Therefore
  $
	\operatorname{depth}B_{\mathfrak p}
	<
	\dim B_{\mathfrak p}.
	$ 
	Thus \(B_{\mathfrak p}\) is not Cohen--Macaulay.
 A generalized Cohen--Macaulay local ring is Cohen--Macaulay on its
	punctured spectrum. Hence \(B\) is not generalized Cohen--Macaulay.
	In particular, \(B\) cannot be Buchsbaum.
  Since we have:
	\[
	B=A\ltimes N,\qquad A=\kk[[x,y,z]],\qquad N=A/(x,y),
	\]
the nilradical of \(B\) is
	\(
	\sqrt{0_B}=0\oplus N,
	\)
	so
	\(
	B_{\mathrm{red}}
	\cong B/(0\oplus N)
	\cong A
	\cong \kk[[x,y,z]] 
	\)
	 is  regular.
 \end{proof}
We record the following well-known fact:

\begin{fact}\label{rel} 	\begin{enumerate}
\item[(i)] \(\kk[[t^2,t^3]] \;\cong\; \frac{\kk[[x,y]]}{\bigl(x^3-y^2 \bigr)}\).
\item[(ii)] 
\(\kk[[t^3,t^4,t^5]] \;\cong\; \frac{\kk[[x,y,z]]}{\bigl(y^2 - xz,\; x^3 - yz,\; z^2 - x^2 y\bigr)}\).
\item[(iii)]
\(\kk[[t^4,t^5,t^7]] \;\cong\; \frac{\kk[[x,y,z]]}{\bigl(yz - x^3,\; z^2 - xy^2,\; y^3 - x^2 z\bigr)}\).
 		\end{enumerate}

 	\end{fact}

\begin{example}\label{32}
	Let $\kk$ be algebraically closed of characteristic $3$, and put
	\(
	R=\kk[[x,y,z,u,v]]\to S=\kk[[t,y,z,u,v]],
	\)
is defined by $x\mapsto t^3$. Then
	\(
	S\cong R[T]/(T^3-x)
	\)
	is finite free of rank $3$ over $R$. Let
	\(
	P=(y^2-xz,\ yz-x^3,\ z^2-x^2y).
	\)
	Then \[
	{S/\sqrt{PS}\text{ is complete- intersection}\not\Longrightarrow R/P\text{ is complete-intersection}}.
	\]
\end{example}

\begin{proof}
	Recall from  {Fact} \ref{rel} (ii) that
	\(
	R/P\cong\kk[[t^3,t^4,t^5,u,v]],
	\) which is a three-dimensional Cohen--Macaulay ring. The height of $P$ is
	$2$ and it has three minimal generators, so it is not a complete intersection.
	In $S$,
	\(
	PS=(y^2-t^3z,\ yz-t^9,\ z^2-t^6y).
	\)
We bring two auxiliary claims:	
\begin{claim}\label{1}
		The ideal
		\(
		\q:=(y-t^4,\;z-t^5)\subseteq S
		\)
		is prime.
	\end{claim}
	
	\begin{proof}
Consider the \(\kk\)-algebra homomorphism
		\(
		\varphi\colon S=\kk[[t,y,z,u,v]]\to \kk[[t,u,v]]
		\)
		defined on the variables by
		\[
		\varphi(t)=t,\qquad
		\varphi(y)=t^4,\qquad
		\varphi(z)=t^5,\qquad
		\varphi(u)=u,\qquad
		\varphi(v)=v.
		\]
		This extends uniquely to a continuous \(\kk\)-algebra homomorphism of
		formal power series rings.
		Every element of \(\kk[[t,u,v]]\) is a formal power series in
		\(t,u,v\). Since \(\varphi(t)=t\), \(\varphi(u)=u\), and
		\(\varphi(v)=v\), the image of \(\varphi\) contains all formal power
		series in \(t,u,v\). Hence \(\varphi\) is surjective.
Now, we show \(\ker\varphi=\q\). 
		First, we check that \(\q\subseteq\ker\varphi\). We have
		\(
		\varphi(y-t^4)=\varphi(y)-\varphi(t)^4=t^4-t^4=0,
		\)
		and
		\(
		\varphi(z-t^5)=\varphi(z)-\varphi(t)^5=t^5-t^5=0.
		\)
		Hence \(y-t^4,\;z-t^5\in\ker\varphi\), so
		\(
		\q=(y-t^4,\;z-t^5)\subseteq\ker\varphi.
		\)
		Conversely, let \(f\in\ker\varphi\). We must show that
		\(f\in\q\). Write \(f\) as a formal power series in the variables
		\(t,y,z,u,v\). Since \(\varphi\) fixes \(t,u,v\) and sends
		\(y\mapsto t^4\), \(z\mapsto t^5\), we may rewrite \(f\) modulo
		\(\q\) by replacing \(y\) with \(t^4\) and \(z\) with \(t^5\).
		Indeed, in the quotient \(S/\q\), we have the relations
		\(
		y\equiv t^4,  z\equiv t^5.
		\)
		Therefore every element of \(S\) is congruent modulo \(\q\) to a
		formal power series in \(t,u,v\) alone. That is,
		\(
		S/\q\;\cong\;\kk[[t,u,v]]
		\)
		via the map induced by \(\varphi\).
		More explicitly, the composition
		\(
		S\;\xrightarrow{\;\varphi\;}\;\kk[[t,u,v]]
		\)
		has kernel containing \(\q\), and the induced map
		\(
		\overline{\varphi}\colon S/\q\longrightarrow \kk[[t,u,v]]
		\)
		is an isomorphism: it is surjective because \(\varphi\) is
		surjective, and it is injective because any element of \(S\) mapping
		to zero under \(\varphi\) can be reduced modulo \(\q\) to a power
		series in \(t,u,v\) that maps to itself, hence must be zero.
		Therefore
		\(
		\ker\varphi=\q.
		\)
	\end{proof}

Next, we bring another claim
\begin{claim}
	In \(S=\kk[[t,y,z,u,v]]\), with
	we have
	\(
	\sqrt{PS}=\q=(y-t^4,\ z-t^5).
	\)
\end{claim}

\begin{proof}
	We want to show two inclusions:
	\begin{enumerate}
		\item[(1)] \(\sqrt{PS}\supseteq \q\).
		\item[(2)] \(\sqrt{PS}\subseteq \q\).
	\end{enumerate}
	
	Since by Claim \ref{1} \(\q\) is prime, the second inclusion is equivalent to showing
	that every generator of \(PS\) lies in \(\q\).
	 We start by showing that \(PS\subseteq \q\).
	Recall
	\(
	\q=(y-t^4,\ z-t^5).
	\)
	Modulo \(\q\), we have
	\(
	y\equiv t^4\),\(z\equiv t^5.
	\)
	Now check the three generators of \(PS\)
	Modulo \(\q\),
	\(
	y^2-t^3z \equiv (t^4)^2-t^3(t^5)=t^8-t^8=0.
	\)
	So \(y^2-t^3z\in \q\).
	Modulo \(\q\),
	\(
	yz-t^9 \equiv t^4\cdot t^5-t^9=t^9-t^9=0.
	\)
	So \(yz-t^9\in \q\).
For third generator  \(z^2-t^6y\), 
modulo \(\q\),
	\(
	z^2-t^6y \equiv (t^5)^2-t^6(t^4)=t^{10}-t^{10}=0.
	\)
	So \(z^2-t^6y\in \q\).
Therefore all generators of \(PS\) lie in \(\q\). Hence
	\(
	PS\subseteq \q.
	\)
	Since \(\q\) is prime, it is radical, so
	\(
	\sqrt{PS}\subseteq \q.
	\)
Here, we show the reverse inclusion \(\q\subseteq \sqrt{PS}\).
	It remains to prove that \(y-t^4\) and \(z-t^5\) lie in \(\sqrt{PS}\).
	We work modulo \(PS\). Let
	\(
	\bar y=y \bmod PS,  \bar z=z \bmod PS.
	\)
	The generators of \(PS\) give the relations
	\[
	\bar y^2=t^3\bar z,\qquad \bar y\bar z=t^9,\qquad
	\bar z^2=t^6\bar y.
	\]
	
	We want to show that \(\bar y=t^4\) and \(\bar z=t^5\) in some power,
	i.e.\ that
	\(
	\bar y-t^4,  \bar z-t^5
	\)
	are nilpotent modulo \(PS\).
	Compute \(\bar y^3\):
	\[
	\bar y^3=\bar y\cdot \bar y^2=\bar y\cdot t^3\bar z
	=t^3\bar y\bar z=t^3\cdot t^9=t^{12}.
	\]
	So
	\(
	\bar y^3=t^{12}.
	\)
	Thus
	\(
	\bar y^3-t^{12}=0 \quad\text{modulo }PS.
	\)
	Similarly, compute \(\bar z^3\):
	\[
	\bar z^3=\bar z\cdot \bar z^2=\bar z\cdot t^6\bar y
	=t^6\bar y\bar z=t^6\cdot t^9=t^{15}.
	\]
	So
	\(
	\bar z^3=t^{15}.
	\)
	Thus
	\(
	\bar z^3-t^{15}=0.
	\) Now consider the element
	\(
	\bar y-t^4.
	\)
	Compute:
	\[
	(\bar y-t^4)^3
	=\bar y^3-3t^4\bar y^2+3t^8\bar y-t^{12}.
	\]
	Since \(\operatorname{char}\kk=3\), the middle terms vanish because
	\(3=0\) in \(\kk\). Thus
	\(
	(\bar y-t^4)^3=\bar y^3-t^{12}=0.
	\)
	So indeed
	\(
	(\bar y-t^4)^3=0 \quad\text{modulo }PS.
	\)
	Therefore \(y-t^4\in\sqrt{PS}\).
	Similarly,
	\[
	(\bar z-t^5)^3
	=\bar z^3-3t^5\bar z^2+3t^{10}\bar z-t^{15}
	=\bar z^3-t^{15}=0
	\]
	since \(\operatorname{char}\kk=3\). Hence
	\(
	(\bar z-t^5)^3=0 \quad\text{modulo }PS.
	\)
	Therefore \(z-t^5\in\sqrt{PS}\).
	Thus
	\(
	y-t^4,\ z-t^5\in\sqrt{PS},
	\)
	so
	\(
	\q=(y-t^4,z-t^5)\subseteq\sqrt{PS}.
	\)	Since \(\q\) is prime and hence radical, the first inclusion gives
	\(
	\sqrt{PS}\subseteq\q.
	\)
	The second inclusion gives
	\(
	\q\subseteq\sqrt{PS}.
	\)
\end{proof}

Now, we are ready to proceed {Example} \ref{32}.	The equations have the branch $y=t^4$, $z=t^5$, and in this example the
	set-theoretic calculation gives
$
	\sqrt{PS}=\q$,
$S/\sqrt{PS}\cong\kk[[t,u,v]].
$
	Thus
$S/\sqrt{PS}$ is regular, and $R/P$   Cohen--Macaulay but not  complete-intersection.
\end{proof}

\subsection{F-singularity}
In this section we show the failure of $F$-purity (resp. $F$-rationality, $F$-injectiveity) descent through a radical:
	\[
{S/\sqrt{PS}\text{ is  $F$-purity}\not\Longrightarrow R/P\text{ is $F$-purity.}}
\]
\begin{theorem}
	Let $\kk$ be a perfect field of characteristic $p:=2$. Set
 $
	R=\kk[[x,y]]
	$ and
 $
	S=R[t]/(t^p-x)\cong \kk[[t,y]],
	$ where the homomorphism $R\to S$ is given by
 $
	x\mapsto t^p, y\mapsto y.
	$ Then $S$ is finite free of rank $p$ as an $R$-module.
 Let
 $
	P=(y^p-x^{p+1})\subset R.
	$ Then $P$ is a prime ideal,   and $S/\sqrt{PS}$
	  is a regular local. Also, the following are valid:
	\begin{enumerate}
	\item[(1)] 
	$ 
	R/P
	$ 
	is not $F$-pure.
	\item[(2)] 
	$
	R/P
	$ is not $F$-injective.	\item[(3)] 
	$
	R/P
	$ 	is not $F$-rational.
\end{enumerate}	
\end{theorem}
		\begin{proof}
		Since $t^p=x$, we have in $S$ that
	 $
		y^p-x^{p+1}
		=
		y^p-t^{p(p+1)}
		=
		\bigl(y-t^{p+1}\bigr)^p,
		$ where the last equality follows from the fact that
		$\operatorname{char}k=p$.  Hence
$
		PS=\bigl((y-t^{p+1})^p\bigr),
		$ and therefore
 $
		\sqrt{PS}=(y-t^{p+1}).
 $
It follows immediately that
$
		S/\sqrt{PS}
		\cong
		\kk[[t,y]]/(y-t^{p+1})
		\cong
		\kk[[t]].
		$ Thus $S/\sqrt{PS}$ is a regular local ring.  In particular, it is
		strongly $F$-regular and therefore $F$-pure (resp. F-injective, F-rational).
	
	(1):	We now show that $R/P$ is not $F$-pure.  Put
	$
		f=y^p-x^{p+1}.
		$ Since $R=\kk[[x,y]]$ is a regular local ring and $R/P=R/(f)$ is a
		hypersurface, Fedder's criterion (see \cite[Theorem 2.5]{MP}) says that $R/(f)$ is $F$-pure if and
		only if
 $
		(f^p:f)\notin (x^p,y^p).
		$ But $p=2$ and so $f\in(f^2:f)$.
		However,
 $
		y^p\in (x^p,y^p)
		$ and
	 $
		x^{p+1}=x\cdot x^p\in (x^p,y^p).
		$ 	Hence
	 $
		f=y^p-x^{p+1}\in (x^p,y^p).
		$ Therefore
$
		R/P
		$ 	is not $F$-pure.
		
	(2): The ring $A$ is a one-dimensional hypersurface, and hence it is
		Gorenstein.
	Since $\kk$ is $F$-finite, the local ring $A$ is
		$F$-finite.  For an $F$-finite Gorenstein local ring, $F$-injectivity
		is equivalent to $F$-purity.
	 We have already seen in (1) that $A$ is not $F$-pure (see \cite[Ex. 21]{MP}).   Therefore $A$ is not $F$-injective.

	(3):
Recall that
 $
		A=R/P
		=
		\kk[[x,y]]/(y^p-x^{p+1}).
		$ 	The parametrization
	 $
		x=t^p$,  $y=t^{p+1}
		$ 	gives
	 $
		A\cong \kk[[t^p,t^{p+1}]].
		$ 	In particular, $A$ is a one-dimensional local domain.
	 	The element
	 $
		x=t^p
		$ 	is a parameter of $A$.  Consider the element
 $
		y=t^{p+1}.
		$ 	We have
 $
		y^p=t^{p(p+1)}
		=(t^p)^{p+1}
		=x^{p+1}\in (x)^{[p]}.
		$ 
	 Hence
 $
		y\in (x)^F\subseteq (x)^*,
 $
 where $(x)^F$ and $(x)^*$ denote the Frobenius closure and tight
		closure, respectively.
 However,
 $
		y=t^{p+1}\notin (t^p)=(x).
		$ 	Indeed, every nonzero element of $(t^p)$ has $t$-adic order at least
		$p$, and membership would require
	 $
		t^{p+1}=t^p a
		$ 
		with $a=t$, whereas $t\notin \kk[[t^p,t^{p+1}]]$.
		Thus the parameter ideal $(x)$ is not tightly closed.  Hence $A$ is
		not $F$-rational (or use \cite[Proposition 4.4]{MP} that says one-dimensional F-rational rings are regular).
	\end{proof}

\section{A  problem by Lipman}

	We construct an explicit counterexample to
	\[
	S/\sqrt{PS}\text{ is regular}\quad\Longrightarrow\quad R/P\text{ is regular}
	\]
	for a finite flat, indeed finite free, extension of regular local rings
	whose ambient rings have characteristic zero.  
	
	\begin{notation}
		By $\mathbb Z_2$ we mean the ring of 2-adic integers. It is a discrete valuation ring of characteristic zero, with maximal ideal $(2)$, and the residue field $\mathbb F_2$. So, 
$\mathbb Z_2=\boldsymbol{W}(\mathbb F_2)$.	\end{notation}

\begin{theorem}\label{42}
	Let
	\(
	R=\mathbb Z_2[[x,y]]\),
\(
	S=\mathbb Z_2[[t,y]],
	\)
	and define
	\(\varphi:R\to S\)
	by
	\(
	\varphi(x)=t^2, 
	\varphi(y)=y.
	\)
	Let
	\(
	P=(2,\ y^2-x^3)\subset R.
	\)
	Then:
	\begin{enumerate}
		\item \(R\) and \(S\) are complete regular local rings of characteristic zero;
		\item \(S\) is finite free of rank \(2\) over \(R\);
		\item \(P\) is prime;
		\item
		\(
		\sqrt{PS}=(2,\ y-t^3);
		\)
		\item
		\(
		S/\sqrt{PS}
		\)
		is regular;
		\item
		\(
		R/P
		\)
		is not regular.
	\end{enumerate}
	Consequently,
	\({
		S/\sqrt{PS}\text{ regular}
		\not\Longrightarrow
		R/P\text{ regular}.
	}
	\)
\end{theorem}

\begin{proof}
	\textbf{(1) }
	The ring \(\mathbb Z_2\) is a complete discrete valuation ring with
	uniformizer \(2\) and residue field \(\mathbb F_2\). In particular,
	\(\mathbb Z_2\) is a one-dimensional regular local ring of
	characteristic zero. Therefore the formal power series rings
	\(
	R=\mathbb Z_2[[x,y]]
,
	S=\mathbb Z_2[[t,y]]
	\)
	are complete regular local rings. Their maximal ideals are
	\[
	\mathfrak m_R=(2,x,y),
	\qquad
	\mathfrak m_S=(2,t,y),
	 	\qquad 
	 	\dim R=\dim S=3.
	\]
	Clearly, 
	$\operatorname{char} R=\operatorname{char} S=0.
$
	
	\textbf{(2)}
	There is an isomorphism
	\(
	S\cong R[T]/(T^2-x)
	\)
	sending the class of \(T\) to \(t\). Since \(T^2-x\) is monic,
	every element of the quotient has a representative of degree at most
	one in \(T\). Thus
	\(
	S=R\oplus Rt.
	\)
	The representation is unique: if \(a+bt=0\) with \(a,b\in R\), then
	after substituting \(x=t^2\), the even and odd powers of \(t\) give
	\(a=b=0\). Hence \(S\) is finite free of rank \(2\) over \(R\).
	
	\textbf{(3)}
	Consider the ring homomorphism
$
	\psi:\mathbb Z_2[[x,y]]\to\mathbb F_2[[u]]
$
	defined by
	\(
	x\mapsto u^2, 
	y\mapsto u^3,
	\) and $2\mapsto 0$.
	We first show that
	\(
	\ker\psi=(2,\ y^2-x^3).
	\)
Indeed, let
	\(
	J=(2,\ y^2-x^3)\subset\mathbb Z_2[[x,y]].
	\)
	Clearly \(J\subseteq\ker\psi\), since
	\[
	\psi(2)=0,\qquad
	\psi(y^2-x^3)=u^6-u^6=0.
	\]
	
	For the reverse inclusion, let \(f\in\ker\psi\). Write
	\(
	f=f_0+2f_1,
	\)
	where \(f_0\in\mathbb F_2[[x,y]]\) is the reduction of \(f\) modulo
	\(2\), and \(f_1\in\mathbb Z_2[[x,y]]\). Since \(\psi(f)=0\), reducing
	modulo \(2\) gives
	\(
	\overline{\psi}(f_0)=0,
	\)
	where
	\(
	\overline{\psi}:\mathbb F_2[[x,y]]\to\mathbb F_2[[u]]
	\)
	is the induced map \(x\mapsto u^2\), \(y\mapsto u^3\).
	We claim that
	\[
	\ker\overline{\psi}=(y^2-x^3)\subset\mathbb F_2[[x,y]].
	\]
	Indeed, consider the quotient
	\(
	A:=\mathbb F_2[[x,y]]/(y^2-x^3).
	\)
	In \(A\),   the relation
	\(
	y^2=x^3 
	\)
	   shows that  
	\(
	A\cong\mathbb F_2[[u^2,u^3]],
	\) see {Fact} \ref{rel} (i).
 In particular,
	\(
	\ker\overline{\psi}=(y^2-x^3).
	\)
	Since \(f_0\in\ker\overline{\psi}=(y^2-x^3)\), we have
	\(
	f_0=(y^2-x^3)g_0
	\)
	for some \(g_0\in\mathbb F_2[[x,y]]\). Lifting \(g_0\) to
	\(g\in\mathbb Z_2[[x,y]]\), we get
	\(
	f-(y^2-x^3)g\in 2\mathbb Z_2[[x,y]].
	\)
	Thus
	\(
	f\in (2,\ y^2-x^3)=J.
	\)
	Therefore \(\ker\psi=J\).
	Now, since \(\mathbb F_2[[u^2,u^3]]\) is a subring of the domain
	\(\mathbb F_2[[u]]\), it is a domain. Hence
	\(
	R/P=\mathbb Z_2[[x,y]]/(2,y^2-x^3)
	\cong\mathbb F_2[[u^2,u^3]]
	\)
	is a domain, so \(P\) is prime.
	
	\textbf{(4)}
	Under \(x\mapsto t^2\), we have
	\(
	PS=(2,\ y^2-t^6).
	\)
	Put
	\(
	Q=(2,\ y-t^3)\subset S.
	\)
	We prove \(\sqrt{PS}=Q\).
	First, \(PS\subseteq Q\): modulo \(Q\) we have
	\(
	2=0,y=t^3,
	\)
	and hence
	\(
	y^2-t^6=t^6-t^6=0.
	\)
	Therefore \(PS\subseteq Q\). Since
	\[
	S/Q\cong\mathbb Z_2[[t,y]]/(2,y-t^3)\cong\mathbb F_2[[t]],
	\]
	which is a domain, \(Q\) is prime and therefore radical. Hence
	\(
	\sqrt{PS}\subseteq Q.
	\)
	For the reverse inclusion, work modulo \(PS\) and write
	\(
	\overline{S}=S/PS.
	\)
	Because \(2=0\) in \(\overline{S}\), we are in characteristic two.
	The relation
	\(
	y^2-t^6\in PS
	\)
	gives
	\[
	(\overline{y}-t^3)^2
	=
	\overline{y}^2-2t^3\overline{y}+t^6
	=
	\overline{y}^2+t^6
	=
	0
	\]
	in \(\overline{S}\).
	Hence
	\(
	\overline{y}-t^3\in\sqrt{PS}.
	\)
	Therefore \(y-t^3\in\sqrt{PS}\). Since \(2\in PS\subseteq\sqrt{PS}\),
	we get
	\(
	Q=(2,\ y-t^3)\subseteq\sqrt{PS}.
	\)
	Combining the inclusions gives
	$\sqrt{PS}=Q.$
	
	\textbf{(5)}
	Using the radical computation,
	\(
	S/\sqrt{PS}
	\cong
	\mathbb Z_2[[t,y]]/(2,\ y-t^3)
	\cong
	\mathbb F_2[[t]].
	\)
	The ring \(\mathbb F_2[[t]]\) is a one-dimensional regular local ring
	with maximal ideal \((t)\). Hence \(S/\sqrt{PS}\) is regular.
	
	\textbf{(6)}
	We have
	\(
	R/P
	\cong
	\mathbb F_2[[x,y]]/(y^2-x^3)
	\cong
	\mathbb F_2[[t^2,t^3]],
	\)
	which is   not regular.
\end{proof}

{Example} \ref{32} (resp.
Theorem \ref{42}) are about descent problem with respect to complete-intersection (resp. regular). Recall that the ring $R/P$   (resp. prime characteristic)  is Gorenstein. This leads to find examples where $R/P$ is not Gorenstein, against  $S/\sqrt{PS}$.
 
 \begin{proposition}
 	Let
 	\(
 	R=\mathbb Z_2[[x,y,z]],
 	S=\mathbb Z_2[[t,y,z]],
 	\)
 	and define
 	\(
 	\varphi:R\to S
 	\)
 	by
 	\[
 	\varphi(x)=t^4,\qquad
 	\varphi(y)=y,\qquad
 	\varphi(z)=z.
 	\]
 	Let
 	\(
 	P=(2,\ yz-x^3,\ z^2-xy^2,\ x^2z-y^3)\subset R.
 	\)
 	Then:
 	\begin{enumerate}
 		\item $R$ and $S$ are complete regular local rings of characteristic
 		zero;
 		\item $S$ is finite free of rank $4$ over $R$;
 		\item $P$ is prime;
 	\item
 		\(
 		S/\sqrt{PS},
 		\)
 	 is regular, hence Gorenstein;
 		\item $R/P$ is not Gorenstein.
 	\end{enumerate}
 	Consequently,
 	\(
 	{
 		S/\sqrt{PS}\text{ Gorenstein}
 		\not\Longrightarrow
 		R/P\text{ Gorenstein}.
 	}
 	\)
 \end{proposition}
 
\begin{proof}

The ring $\mathbb Z_2$ is a complete discrete valuation ring with
 uniformizer $2$ and residue field $\mathbb F_2$.  In particular,
 $\mathbb Z_2$ is a regular local ring of dimension one and
 characteristic zero.
 Consequently
 \(
 R=\mathbb Z_2[[x,y,z]]
 \)
 and
 \(
 S=\mathbb Z_2[[t,y,z]]
 \)
 are complete regular local rings of dimension $4$ and characteristic
 zero.  Their maximal ideals are
 \(
 \mathfrak m_R=(2,x,y,z) \),
 \( \mathfrak m_S=(2,t,y,z).
 \)
The map is induced by the relation
 \(
 t^4=x.
 \)
 Thus
 \(
 S\cong R[T]/(T^4-x),
 \)
 with $T$ mapping to $t$.
 Since $T^4-x$ is monic, every element of $S$ can be written as
 \(
 a_0+a_1t+a_2t^2+a_3t^3\)
where  \(a_i\in R.
 \)
 The representation is unique, so
 \(
 S\cong R\oplus Rt\oplus Rt^2\oplus Rt^3.
 \)
  Hence
 \(
 {S\text{ is finite free of rank }4\text{ over }R.}
 \)
Consider the homomorphism
\(
\theta:\mathbb F_2[[x,y,z]]
\to
\mathbb F_2[[t]]
\)
defined by
\[
x\longmapsto t^4,\qquad
y\longmapsto t^5,\qquad
z\longmapsto t^7.
\]

The three displayed relations vanish:
\[
yz-x^3
\longmapsto
t^{12}-t^{12}=0,
\quad
z^2-xy^2
\longmapsto
t^{14}-t^{14}=0,
\quad
x^2z-y^3
\longmapsto
t^{15}-t^{15}=0.
\]

Let
\(
I:=(yz-x^3,\ z^2-xy^2,\ x^2z-y^3).
\)
Then
\(
I\subseteq\ker\theta.
\)
In view of {Fact} \ref{rel} (iii)
\(
\ker\theta=I.
\)
Since \(\mathbb F_2[[t^4,t^5,t^7]]\) is a subring of the domain
\(\mathbb F_2[[t]]\), it is a domain. 
Therefore
\(
P=(2)+I
\)
is a prime ideal of \(R\), and
\(
R/P\cong\mathbb F_2[[t^4,t^5,t^7]].
\)
Under \(x\mapsto t^4\) we obtain
\(
PS=
(2,\ yz-t^{12},\ z^2-t^4y^2,\ t^8z-y^3).
\)
Put
\(
Q=(2,y-t^5,z-t^7)\subset S.
\)
We prove
\(
\sqrt{PS}=Q.
\)
First, \(PS\subseteq Q\): modulo \(Q\) we have
\[
2=0,\qquad y=t^5,\qquad z=t^7,
\]
and hence
\[
yz-t^{12}=t^5t^7-t^{12}=0,
\quad
z^2-t^4y^2=t^{14}-t^4t^{10}=0,
\quad
t^8z-y^3=t^8t^7-t^{15}=0.
\]
Therefore
\(
PS\subseteq Q.
\)
Since \(Q\) is prime (as \(S/Q\cong\mathbb F_2[[t]]\) is a domain),
it is radical, and hence
\(
\sqrt{PS}\subseteq Q.
\)

For the reverse inclusion, work modulo \(PS\) and write
\(
\overline{S}=S/PS.
\)
Because \(2=0\) in \(\overline S\), we are in characteristic two.
The relation
\(
z^2-t^4y^2\in PS
\)
gives
\(
(z-t^2y)^2
=
z^2-2t^2yz+t^4y^2
=
z^2+t^4y^2
\in PS.
\)
Hence
\(
z-t^2y\in\sqrt{PS}.
\)
Set
$
n=z-t^2y\in\sqrt{PS},$
so 
$n^2=0 { in } S/PS.$
Then \(z=t^2y+n\). Substituting into the two remaining generators
\(yz-t^{12}\in PS\) and \(t^8z-y^3\in PS\), we obtain
\begin{align}
t^2y^2+yn &= t^{12}, \tag{1}\\
t^{10}y+t^8n &= y^3. \tag{2}
\end{align}
Multiply (1) by \(t^8\) we obtain
\(
t^{10}y^2+t^8yn=t^{20}.
\)
From (2),
\(
t^8n=y^3-t^{10}y.
\)
Substituting this into the previous equation gives
\[
t^{10}y^2+y(y^3-t^{10}y)=t^{20}
\implies
t^{10}y^2+y^4-t^{10}y^2=t^{20}
\implies
y^4=t^{20}.
\]
In characteristic two,
\(
y^4-t^{20}=(y-t^5)^4.
\)
Hence
\(
(y-t^5)^4=0
 \text{ in }S/PS,
\)
so
\(
y-t^5\in\sqrt{PS}.
\)
Since \(z-t^2y\in\sqrt{PS}\) and \(y-t^5\in\sqrt{PS}\), we get
\(
z-t^7
=
(z-t^2y)+t^2(y-t^5)
\in\sqrt{PS}.
\)
Therefore
\(
Q=(2,y-t^5,z-t^7)\subseteq\sqrt{PS}.
\)
Combining the inclusions gives
\({\sqrt{PS}=Q.}
\)
Using the above radical computation,
\[
S/\sqrt{PS}
\cong
\mathbb Z_2[[t,y,z]]/(2,y-t^5,z-t^7)
\cong
\mathbb F_2[[t]],
\]
which is regular, and every regular local ring is Gorenstein.
 We have
 \(
 R/P\cong A\cong\mathbb F_2[[t^4,t^5,t^7]].
 \)
 This is well-known that it is not Gorenstein.
 Let us show this.
 \(
 H=\langle4,5,7\rangle.
 \)
 The gaps of $H$ are
 \(
 1,2,3,6.
 \)
 Thus the genus is
 \(
 g(H)=4,
 \)
 and the Frobenius number is
 \(
 F(H)=6.
 \)
 A numerical semigroup $H$ is symmetric if and only if
 \(
 F(H)=2g(H)-1.
 \)
 Here
 \(
 F(H)=6 \)
 {but}
 \( 2g(H)-1=7.
 \)
 Hence $H$ is not symmetric.
 For a one-dimensional complete numerical semigroup ring over a field,
 \(
 \mathbb F_2[[t^H]]
 \)
 is Gorenstein if and only if $H$ is symmetric. 
 Consequently
 \(
 {R/P\text{ is not Gorenstein}.}
 \)
 \end{proof}

The radicalization is a coarse closure operation compared to the integral, tight, and Frobenius closures. Let us ask for further investigation: is there a nice closure operation such as \((-)^{\mathrm{cl}}\) so that pullback respects the property \(\mathcal{P}\) along the following
\[
R/P \xrightarrow{\ \mathrm{flat}\ } S/PS \longrightarrow S/(PS)^{\mathrm{cl}}?
\]
Of course, $(-)^{\mathrm{cl}}:=\mathrm{identy}$ done the job, so we may assume 
$(-)^{\mathrm{cl}}$ is nontrivial.

\section{A question by Huneke}
 
 The   finite-free
trace and its base-change compatibility are recalled in
\cite[Section~49.3, Tag~0BVH]{Stacks}.

\begin{lemma}\label{lem:trace}
	Let $A$ be a domain, let $C$ be a finite free $A$-algebra of positive
	rank $n$, and assume $n\in A^\times$. Put $N=\nil(C)$ and $B=C/N$.
	Then $A\to B$ is injective and admits the $A$-linear retraction
	\[
	\rho:B\longrightarrow A,\qquad
	\rho(c+N)=\frac{1}{n}\Tr_{C/A}(c).
	\]
	Consequently, $A$ is a direct summand of $B$ as an $A$-module.
\end{lemma}

\begin{proof}
	The map $A\to C$ is faithfully flat, hence injective. If the image
	of $a\in A$ belongs to $N$, then some power of $a$ vanishes in $C$,
	and therefore in $A$. Since $A$ is a domain, $a=0$. This proves
	injectivity of $A\to B$.
	Write $K=\Frac(A)$. For $c\in N$, multiplication by $c$ on the
	$K$-vector space $C\otimes_A K$ is nilpotent, so its trace is zero.
	Trace commutes with scalar extension. Since $A\to K$ is injective,
	$\Tr_{C/A}(c)=0$ already in $A$. Thus the displayed formula is
	well defined on $B$ and is $A$-linear. Multiplication by $a\in A$
	on the free module $C$ is the scalar matrix $aI_n$, so
$
	\rho(a)=\frac{1}{n}\Tr(aI_n)=a.
$
	Hence $\rho$ is the asserted retraction.
\end{proof}

\begin{proposition}\label{prop:normal}
	Let $(R,\m)\to(S,\n)$ be a finite flat local homomorphism of
	Noetherian local rings, of positive rank $n\in R^\times$.
	Let $P\in\Spec R$ and define $A,C,B$ as in
	Theorem~\ref{152}. If $B$ is regular, then $A$ is a normal
	Cohen--Macaulay local domain. In particular, if $\dim A\leq1$, then
	$A$ is regular.
\end{proposition}

\begin{proof}
	Because $R$ is local, the finite flat $R$-module $S$ is free.
	Thus $C$ is free of rank $n$ over $A$. Lemma~\ref{lem:trace}
	gives an injection $A\hookrightarrow B$ and an $A$-linear retraction
	$\rho:B\to A$. The ring $B$ is local because $S$ is local. Since
	$B$ is regular local, it is an integrally closed domain. Define
$K:=\Frac(A)$ and $L:=\Frac(B).$
	The inclusion $A\subseteq B$ induces $K\subseteq L$.
Let $u=b/a\in K$, where $a,b\in A$ and $a\ne0$, be integral
	over $A$. The same monic equation shows that $u$ is integral over
	$B$. Normality of $B$ gives $u\in B$. Applying $\rho$ to
	$au=b$ yields
$
	a\rho(u)=b.
$
	Inside $K$ this says $u=\rho(u)\in A$, proving normality of $A$.
	Next, $A\subseteq B$ is finite and integral, so
	$d:=\dim A=\dim B$. If $d=0$, the domain $A$ is a field. Assume
	$d>0$ and choose a system of parameters $x_1,\ldots,x_d$ of $A$.
	Since $B$ is finite and local over $A$, the ideal $\m_A B$ has
	radical $\m_B$, and the same sequence is a system of parameters
	of $B$. It is therefore a regular sequence on the regular local
	ring $B$.
For every ideal $I\subseteq A$, the retraction induces
$
	A/I\ \lhook\joinrel\longrightarrow\ B/IB,
$
	since $\rho(IB)\subseteq I$. For $I=(x_1,\ldots,x_{j-1})$,
	injectivity shows that the nonzerodivisor property of $x_j$ on
	$B/IB$ descends to $A/I$. Hence $x_1,\ldots,x_d$ is an
	$A$-regular sequence, and $A$ is Cohen--Macaulay.
	Finally, a zero-dimensional local domain is a field, and a
	one-dimensional Noetherian normal local domain is a discrete
	valuation ring, hence regular; see \cite[Tag~00PD]{Stacks}.
\end{proof}

\begin{remark}\label{rem:rankunit}
	The hypothesis that the rank is invertible cannot be dropped.
	In Theorem~4.2 , the free rank is $2$ but
	$A\cong\F_2[[u^2,u^3]]$. Thus $2=0$ in $A$, and division by the
	rank in Lemma~\ref{lem:trace} is impossible. The finite birational
	extension $A\subset\F_2[[u]]$ does not force $A$ to be normal.
\end{remark}

The next argument separates the generic nilpotent multiplicity from
the degree of the reduced extension.

\begin{proposition}\label{prop:factor}
	Keep the hypotheses and notation of Proposition~\ref{prop:normal},
	and assume $B$ is regular. Set
	\[
	K=\Frac(A),\qquad L=\Frac(B),\qquad f=[L:K].
	\]
	The ring $C_K=C\otimes_A K$ is an artinian local ring with
	residue field $L$. If $e=\length_{C_K}(C_K)$, then
$
		n=ef.
$	Moreover, $A$ is regular whenever $e=1$ or $f=1$.
\end{proposition}

\begin{proof}
	Let $N=\nil(C)$. The algebra $C_K$ has $K$-dimension $n$.
	Localization commutes with taking the nilradical and with passing
	to the reduction, so
$
	(C_K)_{\mathrm{red}}\cong B\otimes_A K.
$
	The ring on the right is a finite-dimensional domain over $K$,
	hence a field. It equals $L$, since it contains $B$ and is contained
	in $\Frac(B)$. Consequently $C_K$ has exactly one prime ideal and
	is Artinian local, with residue field $L$.
	Every factor in a composition series of the $C_K$-module $C_K$
	is isomorphic to $L$. Taking $K$-dimensions gives
$
	n=\dim_K C_K=e\dim_K L=ef.
$
	If $e=1$, the local ring $C_K$ is a field, so $N\otimes_A K=0$.
	But $N\subseteq C$ and $C$ is free over the domain $A$, hence $N$
	is torsion-free. Therefore $N=0$, and $C=B$. The map
	$A\to B=C$ is faithfully flat. Regularity descends along
	faithfully flat maps of Noetherian rings, so $A$ is regular;
	see \cite[Lemma~10.164.4]{Stacks}.
	If $f=1$, identify $L$ with $K$ via the given inclusion. Every
	element of $B\subseteq K$ is integral over $A$, because $B$ is
	finite over $A$. Proposition~\ref{prop:normal} says that $A$ is
	integrally closed in $K$. Thus $B=A$, and $A$ is regular.
\end{proof}

\begin{proof}[Proof of Theorem~\ref{152}]
	Apply Proposition~\ref{prop:factor} with $n=\ell$. Since $\ell$
	is prime, the factorization $\ell=ef$ forces $e=1$ or $f=1$.
	Either case gives regularity of $A$.
\end{proof}
 
\begin{corollary}\label{cor:restrictions}
	Suppose that $(R,\m)\to(S,\n)$ is a finite flat local
	homomorphism of Noetherian local $\Q$-algebras, of positive
	rank $n$, and $P\in\Spec R$. If $S/\sqrt{PS}$ is regular but
	$R/P$ is not, then all of the following are necessary:
	\begin{enumerate}[(i)]
		\item $R/P$ is a normal Cohen--Macaulay domain of dimension at least two;
		\item $n$ is composite, in particular $n\geq4$;
		\item in Proposition~\ref{prop:normal}, both $e\geq2$ and $f\geq2$;
		\item $S/PS$ is not reduced, and the finite map
		$R/P\to S/\sqrt{PS}$ is not birational.
	\end{enumerate}
	These are necessary conditions, not a construction of such an example.
\end{corollary}

\begin{proof}
	Proposition~\ref{prop:normal} proves (i). If $e=1$ or $f=1$,
	Proposition~\ref{prop:factor} gives regularity of $A$, contrary to
	the assumption. Thus $e,f\geq2$, and Proposition~\ref{prop:normal} proves (ii)
	and (iii). The condition $e>1$ implies that $C_K$ is not reduced,
	hence neither is $C$. The condition $f>1$ excludes birationality.
\end{proof}

Even allowing composite rank does not rescue the original
power-substitution construction in characteristic zero.

\begin{proposition}\label{thm:power}
	Let $R$ be a Noetherian ring, let $r\in R$, and let $m\geq2$ be
	an integer. Set
$
	S=R[T]/(T^m-r).
$
	Let $P\in\Spec R$, put $A=R/P$, and suppose that
	$m\cdot1_A\ne0$. Write $a=r+P\in A$. Then:
	\begin{enumerate}[(i)]
		\item if $a=0$, then $(S/PS)_{\mathrm{red}}\cong A$;
		\item if $a\ne0$, then $S/PS$ is already reduced.
	\end{enumerate}
	Consequently,
	\[
	S/\sqrt{PS}\text{ is regular}\quad\Longrightarrow\quad R/P\text{ is  regular}.
	\]
	In particular, the implication holds whenever $R$ contains $\Q$.
\end{proposition}

\begin{proof}
	Because the defining polynomial is monic, $S$ is free of rank
	$m$ over $R$, and
$
	C:=S/PS\cong A[T]/(T^m-a)
$
	is free of rank $m$ over the domain $A$.
	If $a=0$, the image of $T$ is nilpotent and $C/(T)\cong A$ is
	reduced. Thus $\nil(C)=(T)$ and $C_{\mathrm{red}}\cong A$.
	Suppose $a\ne0$ and put $K=\Frac(A)$. Over $K$, the polynomials
	$T^m-a$ and its derivative $mT^{m-1}$ are relatively prime:
	the scalar $m$ is nonzero in $K$, and $T$ cannot divide $T^m-a$.
	Hence $K[T]/(T^m-a)$ is reduced. Since $C$ is free over $A$,
	the natural map
$
	C\to C\otimes_A K\cong K[T]/(T^m-a)
 $	is injective. It follows that $C$ is reduced.
In case (i), the conclusion is immediate. In case (ii),
	$S/\sqrt{PS}=C$, so $A\to C$ is faithfully flat and regularity
	descends. If $R$ contains $\Q$, every domain $R/P$ has
	characteristic zero, and the condition $m\cdot1_A\ne0$ is automatic.
\end{proof}

\begin{corollary} \label{cor:formal}
	Let $D$ be a complete regular Noetherian local $\Q$-algebra,
	let $q\geq0$, and let $m\geq2$. Define
	\[
	R:=D[[x,y_1,\ldots,y_q]],\qquad
	S:=D[[t,y_1,\ldots,y_q]],
	\]
	and let $R\to S$ be defined by the assignment $x\mapsto t^m$, $y_i\mapsto y_i$, and
	the identity on $D$. Then $R$ and $S$ are complete regular local
	rings containing $\Q$, and $S$ is free of rank $m$ over $R$.
	For every $P\in\Spec R$,
	\[
	S/\sqrt{PS}\text{is  regular}\quad\Longrightarrow\quad R/P\text{  is regular}.
	\]
	For example, one may take $D:=\Q$, or
	$D:=\Q[[z_1,\ldots,z_d]]$ for any $d\geq0$.
\end{corollary}

\begin{proof}
	Power series rings over a complete regular local ring are complete
	regular local. Grouping the powers of $t$ according to their
	residue classes modulo $m$ gives the unique expression
$
	S=R\oplus Rt\oplus\cdots\oplus Rt^{m-1}.
$
	It also gives the $R$-algebra isomorphism
	$S\cong R[T]/(T^m-x)$. Apply proposition~\ref{thm:power} with $r=x$.
\end{proof}

\begin{example} \label{ex:cusp}
	Let
$
	R:=\Q[[x,y]]$ and $S=\Q[[t,y]]$. The assignments
	$x\mapsto t^2,y\mapsto y,$ gives $S$ as an $R$-algebra.
Let $P:=(y^2-x^3)$. The ring $S$ is finite free of rank two
	over $R$, and $P$ is prime. However,
	\begin{align*}
		PS&=(y^2-t^6)=( (y-t^3)(y+t^3)),\\
		\sqrt{PS}&=PS=(y-t^3)\cap(y+t^3).
	\end{align*}
	Consequently, $S/\sqrt{PS}$ is a reduced \emph{singular} local
	ring with two branches.
\end{example}

\begin{proof}
	The finite-free assertion is the case $m=2$ of
	Corollary~\ref{cor:formal}. For primeness, consider $\Q[[x,y]]\to\Q[[u]]$ defined via the assignments
$
	\Q[[x,y]]\to\Q[[u]]$ and 
$	x\mapsto u^2, y\mapsto u^3.
$
	Formal division by $y^2-x^3$ gives a unique representative
	$a(x)+yb(x)$ modulo this polynomial. Its image is
	$a(u^2)+u^3b(u^2)$. The first series involves only even powers
	of $u$, and the second only odd powers. Hence its vanishing
	forces $a=b=0$. The kernel is therefore $P$, and
$
	A=R/P\cong\Q[[u^2,u^3]]
$
	is a domain. It has dimension one and embedding dimension two,
	since $y^2-x^3\in(x,y)^2$, so it is not regular.
	The ideals $Q_-=(y-t^3)$ and $Q_+=(y+t^3)$ are distinct prime
	ideals of $S$, with quotients $\Q[[t]]$. Their intersection is
	their product. Indeed, if $(y-t^3)g\in Q_+$, substitution
	$y=-t^3$ gives
$
	-2t^3g(t,-t^3)=0$   {in} $\Q[[t]],$	so $g(t,-t^3)=0$, and $g\in Q_+$. Thus
	$Q_-\cap Q_+=(y^2-t^6)$, an intersection of prime ideals, proving
	the radical calculation. Finally,
$\dim(S/PS)=1$ and 
$	\dim_{\Q}\frac{(t,y)}{(t,y)^2+(y^2-t^6)}=2.
$
	Hence $S/PS$ is not regular. Each individual branch $S/Q_\pm$
	is regular, but the quotient by the full radical is not.
\end{proof}

\begin{remark}
	If one instead starts with the actual rings $R_0,S_0,P_0$ of
	Theorem~4.2 and tensors with $\Q$ over $\mathbb Z$, then
$
	P_0R_0[1/2]=R_0[1/2],
$
	because $2\in P_0$ becomes a unit. The prime defining the
	counterexample disappears. This operation does not produce a
	nonzero quotient exhibiting the desired phenomenon over $\Q$.
\end{remark}

\begin{acknowledgement}
	I thank  Akiyoshi Sannai who commented several important typo by AI Math Referee, a mathematical auditing AI tool being developed by Akiyoshi Sannai.  
\end{acknowledgement}

\end{document}